\documentclass{amsart}
\usepackage{latexsym,amsmath,amssymb}
\newtheorem{thm}{Theorem}[section]
\newtheorem{cor}[thm]{Corollary}
\newtheorem{exam}[thm]{Example}
\newtheorem{lem}[thm]{Lemma}
\newtheorem{prop}[thm]{Proposition}
\theoremstyle{definition}
\theoremstyle{remark}

\numberwithin{equation}{section}

\begin{document}

\title[]
{On a class of $\lambda$-hyponormal operators}

\author{\sc\bf Y. Estaremi, M. S. Al Ghafri and and S. Shamsigamchi}
\address{\sc Y. Estaremi}
\email{y.estaremi@gu.ac.ir}
\address{Department of Mathematics, Faculty of Sciences, Golestan University, Gorgan, Iran.}
\address{\sc M. S. Al Ghafri}
\address{Department of Mathematics, University of Technology and Applied Sciences, Rustaq {329},  Oman.}
\email{mohammed.alghafri@utas.edu.om}
\address{\sc S. Shamsigamchi}
\address{Department of Mathematics, Payame Noor(PNU) university, Tehran, Iran.}
\email{S.shamsi@pnu.ac.ir}
\address{}
\thanks{}

\thanks{}



\date{}

\dedicatory{}

\commby{}

\begin{abstract}
In this paper we define $\lambda$-hyponormal operators on an infinite dimensional Hilbert space $\mathcal{H}$ and find a class of $\lambda$-hyponormal operators that can not be hypercyclic.  Also, we study closedness of range and $\lambda$-hyponormality of weighted composition operators on the Hilbert space $\mathcal{H}=L^2(\mu)$. Moreover, we apply the hypercyclicity results to $\lambda$-hyponormal weighted composition operators. Finally, we provide some examples to illustrate our main results.

\end{abstract}

\maketitle

\section{ \sc\bf Introduction and Preliminaries}
Let $X$ be a complex topological vector space and $\mathcal{B}(X)$ be the algebra of all continuous linear operators on $X$.
For any $x\in X$ the orbit of $x$ under $T\in \mathcal{B}(X)$ is the set $\{T^n(x):n\geq 0\}$, and is denoted by $Orb(T,x)$.
The continuous operator $T\in \mathcal{B}(X)$ is called {\it hypercyclic} if there exists $x\in X$ such that $Orb(T,x)$ is dense in $X$. Such a vector $x$ is called a {\it hypecyclic} vector for $T$. Also, if there exists a vector $x\in X$ such that the set $\{\lambda T^nx:n\geq 0, \lambda\in \mathbb{C}\}$ is dense in $X$, then $T$ is called {\it supercyclic} and the vector $x$ is called a {\it sepercyclic} vector for $T$. Moreover, $T\in \mathcal{B}(X)$ is called {\it cyclic} if there exists $x\in X$ such that the linear span of $Orb(T,x)$ is dense in $X$. Such a vector $x$ is called a {\it cyclic} vector for $T$.\\
Cyclicity is connected to invariant subspace problem. We say a subspace $M$ of $X$ is invariant under $T\in \mathcal{B}(X)$ if $TM\subseteq M$. The invariant subspace problem asks that a continuous linear operator on a Banach space $X$ has a nontrivial, closed, invariant subspace or not. The ivariant subspace problem is important. It is easy to see that every operator on a finite dimensional complex Banach space with dimension at least two has nontrivial, closed, invariant subspace. But in infinite dimensional Banach spaces the invariant subspace problem is still open. By these observation we understand that a continuous linear operator $T$ on the Banach space $X$ has no nontrivial, closed, invariant subspace if and only if every non-zero vector $x\in X$ is a cyclic vector for $T$. For more information about cyclicity and hypercycility of hyponormal operators one can see \cite{bm, cr, sand}. \\

In this paper we are concern about hypercyclicity of a class of continuous(bounded) linear operators on Hilbert space $\mathcal{H}$, named $\lambda$-hyponormal operators. In section 2 we consider $\lambda$-hyponormal bounded linear operators on infinite dimensional Hilbert space $\mathcal{H}$ and find a class of $\lambda$-hyponormal operators that can not be hypercyclic.  Also, in section 3, we study closedness of range and $\lambda$-hyponormality of weighted composition operators on the Hilbert space $\mathcal{H}=L^2(\mu)$. Moreover, we apply the hypercyclicity results to $\lambda$-hyponormal weighted composition operators.

\section{\sc\bf Hyponormal and $\lambda$-hyponormal operators}

Let $\mathcal{H}$ be an infinite dimensional Hilbert space and $\mathcal{B}(\mathcal{H})$ is the algebra of all bounded linear operators on the Hilbert space $\mathcal{H}$. As is known in the literature, the operator $T\in\mathcal{B}(\mathcal{H})$ is called hyponormal if $T^*T\geq TT^*$. Here for the positive constant $\lambda>0$, we say $T$ is $\lambda$-hyponormal if $T^*T\geq \frac{1}{\lambda}TT^*$. For $\lambda=1$, $1$-hyponormal operators are hyponormal ones. In the following we recall a technical lemma in Banach spaces for later use. 

\begin{lem}[Lemma 4.2,\cite{sand}]\label{l2.1n}
Let $X$ be a Banach space and let $c>1$. If $\{x_n: n\geq 1\}$ is a set in $X$ such that $x_n\geq c^n$ for each $n$, then $\{x_n: n\geq 1\}$ is weakly closed in $X$. 
\end{lem}
In the next proposition we find a sequence of positive numbers corresponding to any $\lambda$-hyponormal operators on an infinite dimensional Hilbert space that play a major role in our main results. 
\begin{prop}\label{p2.1n}
Let $T\in \mathcal{B}(\mathcal{H})$ be a $\lambda$-hyponormal operator for some $\lambda>0$. Then there exists a sequence of positive numbers $\{\lambda_n\}_{n\in\mathbb{N}_0}$ such that for any $h\in \mathcal{H}$, with $h\neq 0$, we have 

$$\|T^nh\|\geq \|h\|\lambda_n\left(\frac{\|Th\|}{\|h\|}\right)^n,$$

for each $n\geq 0$.
\end{prop}
\begin{proof}
Since $T$ is $\lambda$-hyponormal, then by definition we get that for every $h\in \mathcal{H}$,
$\|Th\|\geq \frac{1}{\sqrt{\lambda}}\|T^*h\|$.
Also, as is known, for every non-zero $h\in \mathcal{H}$, we have 
$$\|h\|=\sup_{k\in\mathcal{H},\|k\|\leq 1}|\langle h,k\rangle|.$$
Hence 
\begin{align*}
\|T^2h\|&=\|T(T(h))\|\\
&\geq \frac{1}{\sqrt{\lambda}}\|T^*(Th)\|\\
&=\frac{1}{\sqrt{\lambda}}\sup_{k\in\mathcal{H},\|k\|\leq 1}|\langle T^*(Th),k\rangle|\\
&=\frac{1}{\sqrt{\lambda}}\sup_{k\in\mathcal{H},\|k\|\leq 1}|\langle Th,Tk\rangle|\\
&\geq\frac{1}{\sqrt{\lambda}}|\langle Th,T(\frac{h}{\|h\|})\rangle|\\
&=\frac{1}{\sqrt{\lambda}}\|h\|\left(\frac{\|Th\|}{\|h\|}\right)^2.
\end{align*}
By iterating this method for $n=3,4,5$ we have 
\begin{itemize}
  \item $\|T^3h\|\geq \left(\frac{1}{\sqrt{\lambda}}\right)^3\|h\|\left(\frac{\|Th\|}{\|h\|}\right)^3$
  \item $\|T^4h\|\geq \left(\frac{1}{\sqrt{\lambda}}\right)^6\|h\|\left(\frac{\|Th\|}{\|h\|}\right)^4$
  \item $\|T^5h\|\geq \left(\frac{1}{\sqrt{\lambda}}\right)^{10}\|h\|\left(\frac{\|Th\|}{\|h\|}\right)^{5}$.
\end{itemize}
By these observation and by induction we get that 
$$\|T^nh\|\geq \|h\|\lambda_n\left(\frac{\|Th\|}{\|h\|}\right)^n,$$
in which $\lambda_0=\lambda_1=1$ and $\lambda_{n+1}=\lambda_n \left(\frac{1}{\sqrt{\lambda}}\right)^n$, for $n\geq 0$.
\end{proof}
In the next proposition we will find vectors with weakly closed orbits for a class of $\lambda$-hyponormal operators.
\begin{prop}\label{p2.3n}
Let $T\in \mathcal{B}(\mathcal{H})$ be a $\lambda$-hyponormal operator for some $0<\lambda\leq 1$. If $h\in \mathcal{H}$ such that $\|Th\|>\|h\|$, then $Orb(T,h)$ is weakly closed in $\mathcal{H}$.
\end{prop}
\begin{proof}
Let $0\neq h\in \mathcal{H}$ and $\|Th\|>\|h\|$. Then we set 

$$c^n=\lambda_n\left(\frac{\|Th\|}{\|h\|}\right)^n=\left(\frac{1}{\sqrt{\lambda}}\right)^{t_n}\left(\frac{\|Th\|}{\|h\|}\right)^n,$$

in which $\{\lambda_n\}_n$ is the sequence that is defined in the proof of the Proposition \ref{p2.1n} and $\{t_n\}_n$ is an increasing sequence of positive integers. Since $T$ is $\lambda$-hyponormal, then by the Proposition \ref{p2.1n} we have 
$$\|T^n(\frac{h}{\|h\|})\|\geq c^n.$$
Therefore by Lemma \ref{l2.1n} we get that $Orb(T,\frac{h}{\|h\|})$ is weakly closed in $\mathcal{H}$.  
\end{proof}
Now we find a class of $\lambda$-hyponormal operators that are not weakly hypercyclic.
\begin{thm}\label{t2.4n}
Let $T\in \mathcal{B}(\mathcal{H})$ be a $\lambda$-hyponormal operator for some $0<\lambda\leq 1$. Then $T$ is not weakly hypercyclic.
\end{thm}
\begin{proof}
Suppose on contrary. Let $h\in \mathcal{H}$ be a weakly hypercyclic vector for $T$. We claim that there is some $N\geq 1$ such that $\|T^{N+1}h\|>\|T^Nh\|$, because if $\|T^{n+1}h\|\leq \|T^nh\|$, then $Orb(T,h)$ is norm bounded, but this is impossible, since weakly dense sets cannot be norm bounded. Therefore $\|T^{N+1}h\|>\|T^Nh\|$ for some $N\geq 1$. As $h$ is weakly hypercyclic vector, $T^Nh$ is also a weakly hypercyclic vector. Hence we have a weakly hypercyclic vector $k=T^Nh$ with the property $\|Tk\|>\|k\|$. However, by Proposition \ref{p2.3n} we get that $Orb(T,k)$ is weakly closed. This is a contradiction, because $k$ is weakly hypercyclic vector.
\end{proof}

Here we recall the theorem due to Douglas \cite{Douglas1966} that will be used in the sequel.

\begin{thm}[Douglas's Theorem]\label{t1.1}
    Let \( A, B \in \mathcal{B}(\mathcal{H}) \). Then the following conditions are equivalent:
    \begin{enumerate}
        \item \( \mathcal{R}(B) \subseteq \mathcal{R}(A) \).
        \item There exists \( \lambda > 0 \) such that \( BB^* \leq \lambda A A^* \).
        \item There exists \(C \in \mathcal{B}(\mathcal{H}) \) such that \( AC = B \).
    \end{enumerate}
\end{thm}
\begin{thm} Let $T\in \mathcal{B}(\mathcal{H})$ and $T=T^*C$, for some contraction $C\in \mathcal{B}(\mathcal{H})$ ($\|C\|\leq 1$). Then $T$ can not be weakly hypercyclic.
\end{thm}
\begin{proof}
Since $T=T^*C$ and $C$ is contraction, then by Douglas's Theorem, we have $\|C\|^2T^*T\geq TT^*$ and so $T$ is $\lambda$-hyponormal for $\lambda=\|C\|^2\leq 1$. And so by Theorem \ref{t2.4n} we get that $T$ can not be weakly hypercyclic.
\end{proof}
\section{\sc\bf Weighted composition operators}
Let $(X, \mathcal{F}, \mu)$ be a measure space, $\varphi:X\rightarrow X$ be a non-singular measurable transformation (we mean $\mu\circ\varphi^{-1}\ll\mu$), $h=\frac{d\mu\circ\varphi^{-1}}{d\mu}$ (the Radon-Nikodym derivative of the measure $\mu\circ\varphi^{-1}$ with respect to the measure $\mu$), $Ef=E(f|\varphi^{-1}(\mathcal{F}))$ (the conditional expectation of $f$ with respect to the sigma subalgebra $\varphi^{-1}(\mathcal{F})$ of $\mathcal{F}$) and $C_{\varphi}f=f\circ\varphi$, for every $f\in L^2(\mu)$, be the related composition operator on $L^2(\mu)$.
We recall that for every $f\in L^2(\mu)$,
\begin{itemize}
 \item $C^*_{\varphi}f=hE(f)\circ\varphi^{-1}$
 \item $C^*_{\varphi}C_{\varphi}f=h.f$,
 \item $C_{\varphi}C^*_{\varphi}f=(h\circ\varphi)Ef$
 \item $|C_{\varphi}|f=\sqrt{h}.f$.
\end{itemize}
And if $W=M_uC_{\varphi}$ the weighted composition operator on $L^2(\mu)$, in which $u:X\rightarrow \mathbb{C}$ is a positive measurable function and $P$ is the orthogonal projection of $L^2(\mu)$ onto $\overline{\mathcal{R}(W)}$, then we have
\begin{itemize}
 \item $W^*f=hE(u.f)\circ\varphi^{-1}$,
 \item $W^*Wf=hE(u^2)\circ\varphi^{-1}.f$,
 \item $WW^*f=J\circ\varphi P(f)=u.(h\circ\varphi)E(uf)$,
 \item $|W|f=\sqrt{hE(u^2)\circ\varphi^{-1}}.f$.
\end{itemize}

For $n\in \mathbb{N}$, let $h_n=\frac{d\mu\circ\varphi^{-n}}{\mu}$, $u_n=\prod^{n-1}_{k=0}u\circ\varphi^k$ and 
$$J_n=h_nE_n(|u_n|^2)\circ\varphi^{-n},$$
in which $E_n$ is the conditional expectation operator with respect to the sigma subalgebra $\varphi^{-n}(\mathcal{F})$ of $\mathcal{F}$. In particular we denote $J=J_1$ and $h=h_1$. Since $W^n(f)=u_nf\circ\varphi^n$, for every $n\in\mathbb{N}$ and $f\in L^2(\mu)$, then $W^n$, for $n\geq 2$, is also a weighted composition operator and we have 
\begin{align*}
\int_XJ_n|f|^2d\mu&=\|W^nf\|^2\\
&=\int_X|u_n|^2|f|^2\circ\varphi^nd\mu\\
&\int_Xh_{n-1}E_{n-1}(|u_n|^2|f|^2\circ\varphi^n)\circ\varphi^{-(n-1)}d\mu\\
&\int_Xh_{n-1}E_{n-1}(|u_{n-1}|^2)\circ\varphi^{-(n-1)}|u|^2|f|^2\circ\varphi d\mu\\
&\int_XJ_{n-1}|u|^2|f|^2\circ\varphi d\mu\\
&\int_XhE(J_{n-1}|u|^2)\circ\varphi^{-1}|f|^2d\mu.
\end{align*}
 This implies that
  $$J_n=hE(J_{n-1}|u|^2)\circ\varphi^{-1},$$
 and so for $n\geq 2$ if $u\equiv 1$, then 
 $$h_n=hE(h_{n-1}|u|^2)\circ\varphi^{-1}.$$ For more information about composition and weighted composition operators on $L^2(\mu)$, one can see \cite{bbl, cj}.\\
 Here we rewrite and compute the null space of the weighted composition operators ($W$) and the relation between the null space of $W$ and $W^*$. \cite{bbl, lam}
 \begin{lem}\label{l1.1}
For bounded weighted composition operator $W=M_uC_{\varphi}$ on $L^2(\mu)$ and $J=hE(|u|^2)\circ\varphi^{-1}$, the followings hold:
\begin{itemize}
\item (a) $$\overline{\mathcal{R}(W^*)}=\overline{\mathcal{R}(W^*W)}=\{f\in L^2(\mu): S(f)\subseteq S(J)\}=L^2(S(J)).$$
\item (b) $Ker(W)\subseteq Ker(W^*)$ if and only if $S(u)\subseteq S(J)$. 
\end{itemize}
 \end{lem}
\begin{proof}
(a) Let $f\in L^2(\mu)$. Then $P(f)\in \overline{\mathcal{R}(W)}$ and so there is a sequence $\{f_n\}_n\subseteq L^2(\mu)$ such that $Wf_n\rightarrow P(f)$. Since $Ker(W^*)=\overline{\mathcal{R}(W)}^{\perp}=Ker(P)$, then we have $W^*(f-P(f))=0$, i.e., $W^*(Pf)=W^*(f)$. Also, we know that $W^*W=M_J$ and consequently $S(W^*Wf_n)\subseteq S(J)$. Hence 
$$W^*Wf_n\rightarrow W^*P(f)=W^*f$$ 
and so $S(W^*f)\subseteq S(J)$.\\
For the converse, let $f\in L^2(\mu)$ and $S(f)\subseteq S(J)$. Then we easily can find a sequence $\{f_n\}_n\subseteq L^2(\mu)$ such that $W^*Wf_n\rightarrow f$. Indeed, we can define $f_n=\frac{f.\chi_{A_n}}{J}$, in which $A_n=\{x\in X: J(x)\geq n^{-\frac{1}{2}}\}$ and clearly $\chi_{A_n}\rightarrow \chi_X=1$. Therefore we have
 $$\{f\in L^2(\mu): S(f)\subseteq S(J)\}\subseteq \overline{\mathcal{R}(W^*W)}$$
and then 
 $$\overline{\mathcal{R}(W^*)}\subseteq\{f\in L^2(\mu): S(f)\subseteq S(J)\}\subseteq \overline{\mathcal{R}(W^*W)}.$$
 On the other hands we have $\overline{\mathcal{R}(W^*W)}\subseteq \overline{\mathcal{R}(W^*)}$. By these observations we get the proof.\\
(b) Let $S(u)\subseteq S(J)$ and $f\in \overline{\mathcal{R}(W)}$. Then by definition ($Wf=u.f\circ\varphi$), we have $S(f)\subseteq S(u)\subseteq S(J)$. Hence by part (a) we get that $f\in \overline{\mathcal{R}(W^*)}$, this means that $\overline{\mathcal{R}(W)}\subseteq \overline{\mathcal{R}(W^*)}$ and so $$Ker(W)=\overline{\mathcal{R}(W^*)}^{\perp}\subseteq \overline{\mathcal{R}(W)}^{\perp}=Ker(W^*).$$
 Now let $Ker(W)\subseteq Ker(W^*)$ and $A\subseteq X\setminus S(J)$ such that $\mu(A)<\infty$. Then we get that 
 $$0=\int_AJd\mu=\int_X|u.\chi_A\circ\varphi|^2 d\mu=\int_X|W(\chi_A)|^2=\|W(\chi_A)\|^2.$$
 Hence $\chi_A\in Ker(W)\subseteq Ker(W^*)$, where $Ker(W^*)=\overline{\mathcal{R}(W)}^{\perp}$. Since $X$ is sigma-finite then there is a sequence of measurable sets $\{A_n\}$ with finite measure such that $X=\cup^{\infty}_{n=1}A_n$ and so $X=\cup^{\infty}_{n=1}\varphi^{-1}(A_n)$. As $\mu(A_n)<\infty$, we have $\chi_{A_n}\in L^2(\mu)$ and 
 $$0=\langle W(\chi_{A_n}),\chi_A\rangle=\int_X u.\chi_{A_n}\circ\varphi.\chi_Ad\mu=\int_{\varphi^{-1}(A_n)} u.\chi_Ad\mu.$$
 This implies that $\int_A ud\mu=0$ and so for every measurable subset $B\subseteq A\subseteq X\setminus S(J)$ we have $\int_B ud\mu=0$. Therefore $u\equiv 0$ on $A$ and so $S(u)\subseteq S(J)$.
\end{proof}
In the sequel, we prove that the closedness of range of $W$ on the Hilbert space $L^2(\mu)$ is equivalent to the closedness of range of $W^*$.
\begin{prop}\label{p1.2}
Let $W=M_uC_{\varphi}$ be a bounded operator on $L^2(\mu)$ and $J=hE(|u|^2)\circ\varphi^{-1}$. Then $W$ has closed range if and only if $W^*$ has closed range if and only if $J\geq \delta$,a.e., on $S(J)$, for some $\delta>0$.
\end{prop}
\begin{proof}
For $f\in L^2(\mu)$ we have 
\begin{align*}
\|Wf\|^2&=\int_X|uf\circ\varphi|^2d\mu\\
&=\int_X|u|^2|f|^2\circ\varphi d\mu\\
&=\int_XhE(|u|^2)\circ\varphi^{-1}|f|^2d\mu\\
&=\int_XJ|f|^2d\mu\\
&=\|M_{J^{\frac{1}{2}}}f\|^2.
\end{align*}
This implies that $W$ has closed range if and only if $M_{J^{\frac{1}{2}}}$ has closed range on $L^2(\mu)$. On the other hands, we have $W^*W=M_J$. Also, by part (a) of the Lemma \ref{l1.1} we get that $W^*$ has closed range if and only if $W^*W$ has closed range. Hence $W^*$ has closed range if and only if $M_J$ has closed range. As is known in the literature \cite{tak}, the multiplication operator $M_J$ on $L^2(\mu)$ has closed range if and only if $J\geq \delta$, on $S(J)$, for some $\delta>0$. By these observations we get the results.
\end{proof}
As we know the powers of weighted composition operators are again weighted composition operators. By considering this property we find that the powers of a closed range weighted composition operators have closed range too.
\begin{thm}
If $W$ is a bounded weighted composition operator on $L^2(\mu)$ with closed range and $\varphi^{-1}(S(J))\subseteq S(J)$, then there exists $\delta>0$ such that $J_n\geq \delta^n$, on $S(J)$, for all $n\in \mathbb{N}$. In case $S(J)=X$, for all $n\in \mathbb{N}$, the operator $W^n$ and equivalently $W^{*^n}$ have closed range.
\end{thm}
\begin{proof}
If $W$ has closed range, then by the Proposition \ref{p1.2} we get that 
$$J_1=J\geq \delta, \ \ \mu \  \text{,a.e.,}\ \ \ \ \text{on}\ \ S(J), \ \ \text{for some} \ \ \delta>0.$$
Let $A\in \mathcal{F}$ and $A\subseteq S(J)$, then 
\begin{align*}
\int_AJ_2d\mu&=\int_X \chi_AhE(J|u|^2)\circ\varphi^{-1}d\mu\\
&=\int_X (\chi_{A}\circ\varphi) J|u|^2d\mu\\
&\geq \delta\int_X(\chi_A\circ\varphi)|u|^2d\mu\\
&=\delta\int_AJd\mu\\
&\geq\int_A\delta^2d\mu.
\end{align*}
Hence $J_2\geq \delta^2$, $\mu$, a.e., on $S(J)$. 
  
Similarly and by induction we get that $J_n\geq \delta^n$, a.e., $\mu$, on $S(J)$, for all $n\in \mathbb{N}$ and in case $S(J)=X$, by the Proposition \ref{p1.2}, the operator $W^n=M_{u_n}C_{\varphi^n}$ and equivalently $W^{*^n}=W^{n^*}$ have closed range.
\end{proof}
Now in the next theorem we characterize closed range $\lambda$-hyponormal weighted composition operators on $L^2(\mu)$.
\begin{thm}\label{t3.4}
Let the weighted composition operator $W=M_uC_{\varphi}$ has closed range on $L^2(\mu)$ and $J=hE(|u|^2)\circ\varphi^{-1}$. Then $W$ is $\lambda$-hyponormal for some $\lambda>0$ if and only if $S(u)\subseteq S(J)$.
\end{thm}
\begin{proof} As is known we have 
$$Ker(W)=\overline{\mathcal{R}(W^*)}^{\perp}, \ \ \ \ Ker(W^*)=\overline{\mathcal{R}(W)}^{\perp}.$$
Hence $Ker(W)\subseteq Ker(W^*)$ if and only if $Ker(W^*)^{\perp}\subseteq Ker(W)^{\perp}$ if and only if $\overline{\mathcal{R}(W)}\subseteq \overline{\mathcal{R}(W^*)}$. On the other hands the by the Proposition \ref{p1.2} and the assumption that $W$ has closed range, we get that $Ker(W)\subseteq Ker(W^*)$ if and only if $\mathcal{R}(W) \subseteq \mathcal{R}(W^*)$. By the Lemma \ref{l1.1}, part (b) and the Theorem \ref{t1.1} we get that $S(u)\subseteq S(J)$ if and only if there exists $\lambda>0$ such that $WW^*\leq \lambda W^*W$. Therefore we have the results.

\end{proof}

\begin{thm}\label{t3.5} For $\lambda>0$, the weighted composition operator $W=M_uC_{\varphi}$ is $\lambda$-hyponormal on $L^2(\mu)$ if and only if $S(u)\subseteq S(J)$ and $h\circ\varphi\left(E\left[\frac{u^2}{J}\right]\right)\leq \lambda$.
\end{thm}
\begin{proof}
By making minor changes in the proof of Theorem 3.1 of \cite{lam}, we easily get the proof.
\end{proof}
\begin{cor}\label{cor3.6}
If $S(u)\subseteq S(J)$ and $h\circ \varphi E(\frac{|u|^2}{J})\leq \lambda$, with $0<\lambda\leq 1$, then $W=M_uC_{\varphi}$ is not weakly hypercyclic.
\end{cor}
\begin{proof}
By Theorems \ref{t2.4n} and \ref{t3.5} we get the proof.
\end{proof}

\begin{exam} Let $X=[0, \frac{1}{2}]$, $d\mu=dx$ and $\Sigma$
be the Lebesgue sets. Take $u(x)=\sqrt{x^3}$ and $\varphi(x)=x^2$. 
 It follows that, $\varphi^{-1}(x)=\sqrt{x}$, $h(x)=\frac{1}{2\sqrt{x}}$ and for every $a,b\in [0,1]$ we have 
 $$\int_{\varphi^{-1}(a,b)}f(x)dx=\int_{\sqrt{a}}^{\sqrt{b}}f(x)dx=\int_{a}^{b}\frac{f(\sqrt{x})}{2\sqrt{x}}dx.$$
 This means that $E(f)\circ\varphi^{-1}(x)=\frac{f(\sqrt{x})}{2\sqrt{x}}$. Hence 
 $$J(x)=h(x)E(|u|^2)\circ\varphi^{-1}(x)=\frac{\sqrt{x}}{4}.$$
 By these observations we get that $h\circ \varphi(x) E(\frac{|u|^2}{J})(x)=2\sqrt{x^3}\leq 1$, for all $x\in X$. Hence by Corollary \ref{cor3.6} we have that $W=M_uC_{\varphi}$ is not weakly hypercyclic on the Hilbert space $L^2(\mu)$. Moreover, by Theorem \ref{t3.4} we get that $M_uC_{\varphi}$ is $\lambda$-hypnormal on $L^2(\mu)$, for all $\lambda\geq 1$.

\end{exam}

\begin{exam}
Let $m=\{m_n\}^{\infty}_{n=1}$ 
be a sequence of positive real numbers.
Consider the space $l^2(\mu)=L^2(\mathbb{N},2^{\mathbb{N}}, \mu)$, where $2^{\mathbb{N}}$ is the power set of natural
numbers and $\mu$ is a measure on $2^{\mathbb{N}}$ defined by $\mu(\{n\})=m_n$. Let $u=\{u_n\}^{\infty}_{n=1}$ be a sequence of non-negative real numbers. Let $\varphi:\mathbb{N}\rightarrow \mathbb{N}$ be a non-singular
measurable transformation. Direct computation shows that
$$h(k)=\frac{1}{m_k}\sum_{j\in \varphi^{-1}(k)}m_j, \ \  E^{\varphi}(f)(k)=\frac{\sum_{j\in \varphi^{-1}(\varphi(k))}f_jm_j}{\sum_{j\in \varphi^{-1}(\varphi(k))}m_j},$$
for all non-negative sequence $f=\{f_n\}^{\infty}_{n=1}$ and $k\in \mathbb{N}$. Since $S(J)=S(h)\cap S(E(\varphi(u)))$, then by Theorem \ref{t3.4} we get that $M_uC_{\varphi}$ is $\lambda$-hypnormal on $l^2(\mu)$ if and only if
$S(u)\subseteq \{k\in \varphi(\mathbb{N}):u(\varphi^{-1}(\varphi(k)))\neq\{0\}\}$ and

$$h\circ\varphi(k)E(\frac{u^2}{J})(k)=\frac{1}{m_k}\sum_{j\in \varphi^{-1}(\varphi(k))}\frac{u(j)^2m_j}{J(j)}\leq \lambda$$
on $S(u)$, where for each $j\in \mathbb{N}$,

$$J(j)=\frac{1}{m_j}\sum_{i\in \varphi^{-1}(j)}u(i)^2m_i\leq M$$
for some $M<\infty$.\\
Consequently, if there exists $\lambda\leq 1$ such that 
$$h\circ\varphi(k)E(\frac{u^2}{J})(k)=\frac{1}{m_k}\sum_{j\in \varphi^{-1}(\varphi(k))}\frac{u(j)^2m_j}{J(j)}\leq \lambda$$
on $S(u)$, then $W=M_uC_{\varphi}$ is not weakly hypercyclic on the sequence space $l^2(\mu)$.

\end{exam}

\end{document}